\documentclass[11pt,a4paper]{article}

\usepackage[a4paper,margin = 2.5cm]{geometry}
\usepackage{amsmath,amsthm,amssymb}
\usepackage{mathrsfs}
\usepackage[utf8]{inputenc}
\usepackage{enumerate}
\usepackage{multicol}
\usepackage{graphicx}
\usepackage{color}
\usepackage{xfrac}
\usepackage{bbding}
\usepackage{amsfonts}

\newtheorem{theorem}{Theorem}[section]
\newtheorem{lemma}[theorem]{Lemma}

\newtheorem{corollary}[theorem]{Corollary}

\theoremstyle{remark}
\newtheorem{remark}[theorem]{Remark}
\newtheorem{claim}{Claim} 

\theoremstyle{definition}

\newcommand{\cl}{\text{cl} }
\newcommand{\imp}{\mathbin{\rightarrow}}
\newcommand{\s}[2]{\{ #1 \mapsto #2 \} }

\newcommand{\idp}{\mathrm{Idp}}

\begin{document} 

\title{Strong completeness for the predicate logic of the continuous t-norms}

\author{D. Castaño, J. P. Díaz Varela, G. Savoy}

\maketitle

\begin{abstract}
The axiomatic system introduced by Hájek axiomatizes first-order logic based on BL-chains. 
In this study, we extend this system with the axiom $(\forall x \phi)^2 \leftrightarrow \forall x \phi^2$ and the infinitary rule 
\[
\frac{\phi \vee (\alpha \to \beta^n):n \in \mathbb{N}}{\phi \vee (\alpha \to \alpha \& \beta)}
\] 
to achieve strong completeness with respect to continuous t-norms. 
\end{abstract}

\section{Introduction}

In his 1961 paper \cite{Mosto61}, Mostowski proposed the study of first-order many-valued logics interpreting universal and existential quantifiers as infimum and supremum, respectively, on a set of truth values. 
From the 1963 article by Hay \cite{Hay59}, it follows that the infinitary rule
\[
    \frac{\phi \oplus \phi ^n : n \in \mathbb{N}}{\phi}
\] 
can be added to the first-order Łukasiewicz calculus to obtain weak completeness with respect to the Łukasiewicz t-norm.
Horn \cite{Horn69} later axiomatized first-order Gödel logic in 1969. 

Hájek \cite{Hajek98} provided a general approach to first-order fuzzy logic, introducing a syntactic logic, denoted by BL$\forall$, which is strongly complete with respect to models based on BL-chains. 
However, the problem of finding an appropriate syntactic logic for models based on continuous t-norms remained unresolved. 

In the propositional case, Hájek \cite{Hajek98} exhibited a syntactic logic that is strongly complete with respect to valuations on BL-chains. Ku\l acka \cite{Kulacka18} later proved that by adding the infinitary rule 
\[
    \frac{\phi \vee (\alpha \to \beta^n):n \in \mathbb{N}}{\phi \vee (\alpha \to \alpha \& \beta)}
\] 
to the syntactic logic, a strong completeness result can be achieved with respect to valuations on t-norms. 

In the first-order case, quantifiers can exhibit distinct behaviors in a continuous t-norm compared to a generic BL-chain. 
For example, the sentence,
\[
    \tag{RC}
    \forall x (\phi \& \phi) \to ((\forall x \phi)\&(\forall x \phi ))
\]  
is true in models based on continuous t-norms and is not true in general.
Moreover, Hájek and Montagna \cite{HajekMontagna08} demonstrated that standard first-order tautologies coincide with first-order tautologies over complete BL-chains satisfying (RC).

In this paper we show that by adding Kułacka's Infinitary Rule and Hájek and Montagna's axiom RC to BL$\forall$, a strong standard completeness result can be proven for models based on continuous t-norms.

The structure of this paper is as follows. 
Section \ref{preliminares} compiles pertinent definitions of the logic BL$\forall$, continuous t-norms and related algebraic structures. It also specifies our notion of model, validity and semantic consequence.
Section \ref{definicionLogicaInf} introduces a novel logic, extending the logic BL$\forall$ with an additional axiom and an infinitary rule.
Lemmas describing properties of the disjunction of this new logic are proved. 
A Henkin construction is then utilized to demonstrate that for a given theory $\Gamma$ and a sentence $\phi$ such that $\Gamma \nvdash \phi$, there exists an expanded theory $\Gamma^*$ that also satisfies $\Gamma^* \nvdash \phi$ and possesses additional properties (Henkin property and \textit{prelinearity}) necessary for the subsequent construction of a desirable Lindenbaum algebra. 
In Section \ref{sectionCompleteness}, a Lindenbaum algebra is constructed and embedded into a continuous t-norm, with the help of a weak saturation result also proven in this section, providing the final requisite for the strong completeness result.

\section{Preliminaries}
\label{preliminares}

\subsection{Hájek's basic predicate logic}
A predicate language is a countable set consisting of variables $x, y, z, \ldots$, symbols for constants $c,d, \ldots$, symbols for predicates $P, Q, \ldots$, logical connectives $\to, \&$, the propositional constant $0$, and the symbols $\forall$ and $\exists$. 
Let $\mathcal{L}$ be a predicate language. Terms, formulas and sentences are built in the usual way from $\mathcal{L}$. 
Further connectives are defined as follows:
\begin{itemize}
    \item $\phi \wedge \psi $ is $\phi \& (\phi \to \psi)$,
    \item $\phi \vee \psi $ is $((\phi \to \psi) \to \psi) \wedge ((\psi \to \phi) \to \phi)$,
    \item $\neg \phi $ is $\phi \to 0 $,
    \item $1 $ is $0 \to 0 $.
\end{itemize}
We denote by $\phi \s{x}{t}$ the formula resulting from substituting the term $t$ for the free ocurrences of $x$ in $\phi$. 

The BL$\forall$ logic over $\mathcal{L}$ is the predicate logic given in Hilbert style by the following axioms and inference rules, where $\phi, \psi, \chi$ are formulas:
\begin{itemize}
    \item Propositional axiom schemata:
        \begin{description}
           \item[(A1)] $(\phi \to \psi) \to ((\psi \to \chi) \to (\phi \to \chi))$,
            \item[(A2)] $(\phi \& \psi) \to \phi$,
            \item[(A3)] $(\phi \& \psi) \to (\psi \& \phi)$,
            \item[(A4)] $(\phi \& (\phi \to \psi)) \to (\psi \& (\psi \to \phi))$, 
            \item[(A5)] $(\phi \to (\psi \to \chi)) \to ((\phi \& \psi) \to \chi)$,
            \item[(A6)] $((\phi \& \psi) \to \chi ) \to ( \phi \to (\psi \to \chi))$,
            \item[(A7)] $((\phi \to \psi) \to \chi) \to (((\psi \to \phi) \to \chi) \to \chi)$, 
            \item[(A8)] $0 \to \phi $.
        \end{description}
    \item First-order axiom schemata:
        \begin{description}
            \item[($\forall$1)] $\forall x \phi \to \phi\s{x}{t}$, where $t$ is substitutable for $x$ in $\phi$,
            \item[($\exists$1)] $\phi\s{x}{t} \to \exists x \phi$, where $t$ is substitutable for $x$ in $\phi$, 
            \item[($\forall$2)] $\forall x (\phi \to \psi) \to (\phi \to \forall \psi)$, where $x$ is not free in $\phi$,
            \item[($\exists$2)] $\forall x (\psi \to \phi) \to (\exists x \psi \to \phi)$, where $x$ is not free in $\phi$,
            \item[(Lin)] $\forall x (\psi \vee \phi) \to ((\forall x \psi) \vee \phi)$, where $x$ is not free in $\phi$.
        \end{description}
    \item Inference rules: 
        $$ \text{Modus Ponens:  } \frac{ \phi, \phi \to \psi}{ \phi} $$
        $$ \text{Generalization:  }  \frac{ \phi}{ \forall x \phi}$$
\end{itemize}

Let $\Gamma$ be a set of formulas.
As usual we write $\Gamma \vdash_{BL\forall} \phi$ whenever there is a proof of $\phi$ from $\Gamma$ in BL$\forall$.

\subsection{Continuous t-norms and BL-algebras}

Let $ \cdot : [0,1]^2 \to [0,1]$ be an operation, where $[0,1]$ is the real interval with its usual order. We say $\cdot$ is a \textit{t-norm} if the following conditions hold for $x,y,z \in [0,1]$:
\begin{itemize}
    \item $x \cdot  y = y \cdot  x$,
    \item $x\cdot (y\cdot z)=(x\cdot y)\cdot z$,
    \item $x\cdot y \leq x\cdot z$ if $y\leq z$,
    \item $x\cdot 1 = x$.
\end{itemize}

it is known that if $\cdot$ is continuous (with respect to the usual topology on $[0,1]$):
\begin{itemize}
    \item we can define an operation $\to \colon [0,1]^2 \to [0,1]$, the \textit{residuum} of $\cdot$, by 
        \[
            x \cdot y \leq z \iff x \leq y \imp z
        \]
 for all $x,y \in [0,1]$;   
    \item $ x \wedge y := \min\{ x,y \} = x \cdot (x \to y)$ for all $x,y \in [0,1]$;
    \item $ x \vee y := \max\{ x,y \} = ((x \to y)\to y ) \wedge ((y\to x) \to x)$ for all $x,y \in [0,1]$.
\end{itemize}
Thus, we can define an algebra $([0,1], \cdot, \to, \wedge, \vee, 0, 1)$.
We will also refer to this algebra as a continuous t-norm, and sometimes also denote it as $([0,1], \cdot)$. 

Prominent continuous t-norms include:
\begin{itemize}
    \item the \L ukasiewicz t-norm, defined as $a\cdot b = \max \{  a+b-1, 0 \} $,
    \item the product t-norm, which is the usual real product,
    \item the Gödel t-norm, given by $a \cdot b = \min \{ a,b \}$.
\end{itemize}

The continuous t-norms are an example of complete totally ordered \textit{BL-algebras}. 
We will assume that the reader is familiar with the basic properties of BL-algebras (see for example \cite{Hajek98}).
It is known that the class of all continuous t-norms generates the class of all BL-algebras as a variety \cite{CEGT00}.

We will also make use of the concept of \textit{hoops} in our work, so we will provide a brief introduction to them.

A \textit{hoop} is defined as an algebra $\mathbf{H}=(H,\cdot,\to,1)$ such that $(H,\cdot,1)$ is a commutative monoid and for all $x,y\in H$:
\begin{itemize} 
\item $x\to x =1$.
\item $x \cdot (x \to y) = y \cdot (y \to x)$.
\item $x \to (y \to z) = (x\cdot y ) \to z$.
\end{itemize} 
If a hoop satisfies $(x\to y) \to y = (y \to x) \to x$, it is referred to as a \textit{Wajsberg hoop}.

Let $I$ be a totally ordered set. 
For all $i\in I$ let $\mathbf{H}_i$ be a hoop such that for $i\neq j$, $A_i \cap A_j = \{1\}$.
Then $\bigoplus \mathbf{H}_i$ (the \textit{ordinal sum} of the familiy $(\mathbf{H}_i)_{i\in I}$) is the structure whose base set is $\bigcup H_i$ and the operations are
\begin{equation*}
x\to y = 
\begin{cases}
x \to^{\mathbf{H}_i} y  &\text{if } x,y \in H_i, \\
y                       &\text{if } x \in H_i \text{ and } y\in H_j \text{ with } i>j,\\
1                       &\text{if } x \in H_i-\{1\} \text{ and } y\in H_j \text{ with } i<j,
\end{cases}
\end{equation*}
\begin{equation*}
x\cdot y = 
\begin{cases}
x \cdot^{\mathbf{H}_i} y  &\text{if } x,y \in H_i, \\
y                         &\text{if } x \in H_i \text{ and } y\in H_j-\{1\} \text{ with } i>j,\\
x                         &\text{if } x \in H_i-\{1\} \text{ and } y\in H_j \text{ with } i<j.
\end{cases}
\end{equation*}

A key result that will be particularly useful is the following lemma (see Busaniche \cite{Busaniche2005}):
\begin{lemma} \label{decomposition}
Each BL-chain $\mathbf{A}$ is an ordinal sum of totally ordered Wajsberg hoops. 
\end{lemma}

\subsection{Models and validity}

We are now in a position to define what we will understand to be a model based on a BL-chain. 
Let $\mathbf{B}$ be a BL-chain (a totally ordered BL-algebra).  
A \textit{model based on} $\textbf{B}$, or $\textbf{B}$\textit{-structure}, is a pair $\textbf{M}=(M, .^\textbf{M})$ such that $M$ is a non empty set called  \textit{domain} and $.^\textbf{M}$ is a function that associates with each constant symbol $c$ in $\mathcal{L}$ an element $c^{\textbf{M}}$ in $M$ and with each $n$-ary predicate symbol $P$ a function $P^{\textbf{M}}:M^n \to B$. 

A \textit{valuation} is a function $v$ that maps variables of $\mathcal{L}$ into elements of $M$. If $v$ and $w$ are valuations, and $x$ a variable, we write $v \equiv_x w$ to
indicate that $v$ and $w$ are equal, except, maybe, at the variable $x$. 

Let $\textbf{M}$ be a $\textbf{B}$-structure and $v$ a valuation. 
For each constant $c$, we define $c^{\textbf{M},v}$ as $c^{\textbf{M}}$.
and for each variable $x$, we define $x^{\textbf{M},v}$ as $v(x)$.
If $\phi$ is a formula in the language $\mathcal{L}$, we define its value $\| \phi \|^{\textbf{M},v}$, inductively in the following manner:
\begin{itemize}
    \item if $\phi$ has the form $P(t_1, \ldots, t_n)$, then
        $$
        \| \phi \|^{\textbf{M},v} = P^{\textbf{M}}(t_1^{\textbf{M}, v}, \ldots, t_n^{\textbf{M}, v}),
        $$
    \item $\| 0 \|^{\textbf{M},v} = 0^\mathbf{B}$, $\| 1 \|^{\textbf{M},v} = 1^\mathbf{B}$.
    \item if $\phi= \psi \circ \chi$, where $\circ$ is a logical connective, then $\| \phi \|^{\textbf{M},v}$ is defined iff $\| \psi \|^{\textbf{M},v}$ and $\| \chi
        \|^{\textbf{M},v}$ are defined, and if this is the case, $\| \phi \|^{\textbf{M},v}= \| \psi \|^{\textbf{M},v} \circ ^\mathbf{B} \| \chi \|^{\textbf{M},v}$.
    \item if $\phi= \forall x \psi$, then $\| \phi \|^{\textbf{M},v}$ is defined iff
        \begin{itemize}
            \item for all valuations $w \equiv_x v$, $\| \psi \|^{\textbf{M},w}$ is defined and,
            \item $\bigwedge_{w \equiv_x v} \| \psi \|^{\textbf{M},w}$ exists in $\textbf{B}$. 
        \end{itemize}
        If that is the case, $\| \phi \|^{\textbf{M},v} = \bigwedge_{w \equiv_x v} \| \psi \|^{\textbf{M},w}$.
    \item if $\phi= \exists x \psi$, then $\| \phi \|^{\textbf{M},v}$ is defined iff
        \begin{itemize}
            \item for all valuations $w \equiv_x v$, $\| \psi \|^{\textbf{M},w}$ is defined and,
            \item $\bigvee_{w \equiv_x v} \| \psi \|^{\textbf{M},w}$ exists in $\textbf{B}$. 
        \end{itemize}
        If that is the case, $\| \phi \|^{\textbf{M},v} = \bigvee_{w \equiv_x v} \| \psi \|^{\textbf{M},w}$.
\end{itemize}

If $\phi$ is a sentence, that is, a formula with no free variables, we write $\| \phi \|^{\textbf{M}}$ instead of $\| \phi \|^{\textbf{M},v}$. 
A $\textbf{B}$-structure $\textbf{M}$ is called \textit{safe} if $\| \phi \|^{\textbf{M},v}$ is defined for all formulas $\phi$ and all valuations $v$. 
A safe $\textbf{B}$-structure $\textbf{M}$ is a \textit{model} of a formula $\phi$ provided $\| \phi \|^{\textbf{M},v} = 1$ for all valuations $v$; $\textbf{M}$ is a model of a set $\Gamma$ of formulas if it is a model of each $\gamma \in \Gamma$.

Let $\Gamma$ be a set of formulas, $\phi $ a formula and $\mathcal{K}$ a class of BL-chains. 
We say that $\phi$ \textit{is a semantic consequence of} $\Gamma$ \textit{in} $\mathcal{K}$, and we write $\Gamma \vDash_{\mathcal{K}} \phi$, provided for all $\textbf{B}\in \mathcal{K}$ and for each safe $\textbf{B}$-structure $\textbf{M}$, if $\textbf{M}$ is a model of $\Gamma$, then $\textbf{M}$ is a model of $\phi$. 
We say $\phi$ is \textit{valid in} $\mathcal{K}$  if $\vDash_{\mathcal{K}} \phi$, and that $\phi$ is
\textit{valid in a BL-chain} $\textbf{B}$ if it is valid in $\mathcal{K} = \{ \textbf{B} \}$. 

If $\mathcal{K}$ represents the class of all continuous t-norms, we will use the notation $\vDash$ to refer to $\vDash_{\mathcal{K}}$.

\section{The predicate logic based on t-norms}
\label{definicionLogicaInf}

For the case where $\mathcal{K}$ is the class of all BL-chains, Hájek proved \cite{Hajek98} that the logic BL$\forall$ is strongly complete with respect to models based on $\mathcal{K}$, that is for all $\Gamma$ and all $\phi$ 
\[
    \Gamma \vdash_{BL\forall} \phi \iff \Gamma \vDash_{\mathcal{K}} \phi .
\]

Let $\mathcal{K}$ be the class of all t-norms. 
Let us briefly examine $\vDash_{\mathcal{K}}$. 
Suppose $\alpha$ and $\beta$ are sentences. 
It can be shown that $\{\alpha \imp \beta^n, n\in \mathbb{N}\} \vDash_{\mathcal{K}} \alpha \imp (\alpha \& \beta)$. 
However, there are sentences $\alpha,\beta$ such that for any finite subset $\Gamma$ of $\{ \alpha \imp \beta^n, n\in \mathbb{N} \}$, we have that $\Gamma \not\vDash_{\mathcal{K}} \alpha \to (\alpha \& \beta)$. 
(quizas en el libro de hajek o hay)
This shows that this logic is infinitary. 
Then we will need an infinitary syntactic logic if we want a strong completeness result.

We will define an infinitary logic extending Hájek's basic predicate logic adding an infinitary rule
\begin{equation}
    \tag{Inf}
    \label{Inf}
    \frac{\phi \vee (\alpha \to \beta^n):n \in \mathbb{N}}{\phi \vee (\alpha \to \alpha \& \beta)}
\end{equation}
and an axiom schema
\begin{equation}
    \tag{RC}
    \label{RC}
    \forall x (\chi \& \chi) \to (\forall x \chi) \&( \forall x \chi)     
\end{equation}
where $\phi, \alpha$ and $\beta$ are sentences and $\chi$ a formula.

Kułacka introduced rule \ref{Inf} in \cite{Kulacka18} to demonstrate the strong standard completeness of propositional basic logic.
The axiom \ref{RC} was studied by Montagna and Hájek in \cite{HajekMontagna08} where they exhibited a very close relation between continuous t-norms and models based on complete BL-chains where \ref{RC} is valid. 

We let $\vdash$ denote the extension of Hájek's basic predicate logic by the Infinitary Rule (Inf) and the axiom (RC).  

Let $\Gamma$ be a set of formulas. 
Formally, \textit{a proof from $\Gamma$} is a sequence of formulas $\{\phi_i \}_{i \leq \xi}$ for some ordinal $\xi$ (proof length) such that, for each $i$ with $i \leq \xi$, denoting $\mathcal{H}_i= \{ \phi_j : j<i\} $, at least one of the following conditions holds,
\begin{itemize}
    \item $\phi_i$ is an axiom of $\vdash_{BL\forall}$,
    \item $\phi_i$ is an instance of RC, 
    \item $\phi_i \in \Gamma$,
    \item there exists a formula $\chi$ such that $\chi \in \mathcal{H}_i$ and $\chi \to \phi_i \in \mathcal{H}_i$,
    \item there exists a formula $\chi$ such that $\forall x \chi = \phi_i$ and $\chi \in \mathcal{H}_i$,
    \item there exist formulas $\phi, \alpha, \beta$ such that $ \phi_i = \phi \vee (\alpha \to (\alpha \& \beta))$ and $\phi \vee (\alpha \to \beta^n) \in \mathcal{H}_i $ for each $n \in \mathbb{N}$. 
\end{itemize}

If $\phi$ is a formula, \textit{a proof of $\phi$ from $\Gamma$} is a proof from $\Gamma$ ending in $\phi$. 
We write $\Gamma \vdash \phi$ if there is a proof of $\phi$ from $\Gamma$. 

Given formulas $\phi,\psi$ and sets of formulas $\Gamma, \Delta$ it is easy to prove that $\vdash$ satisfies 
\begin{itemize}
    \item Reflexivity: if $\phi \in \Gamma $, then $\Gamma \vdash \phi $.
    \item Monotonicity: if $\Gamma \vdash \phi$ and $\Gamma \subseteq \Delta$, then $\Delta \vdash \phi$. 
    \item Transitivity: if $\Gamma \vdash \phi$ and $\Delta \vdash \theta$ for all $\theta \in \Gamma$, then $\Delta \vdash \phi$.
    \item if $\Gamma \vdash_{BL\forall} \phi$, then $\Gamma \vdash \phi$.
\end{itemize}

The \textit{lemma on constants} of first-order logic is also true in our logic. 
\begin{lemma}
    If $\Gamma \vdash \phi$ where $c$ is a constant not present in $\Gamma$ and $x$ does not appear free in $\phi $, then $\Gamma \vdash (\forall x)\phi\s{c}{x}$. 
\end{lemma}

\begin{remark}
    Note that, if $\mathcal{L}$ is a predicate language, $\Gamma \cup \{ \phi \}$ is a set of formulas in $\mathcal{L}$, and $\mathcal{L}^\prime$ is an extension of $\mathcal{L}$ obtained by adding infinite new constants, then there is a proof of $\phi$ from $\Gamma$ in the language $\mathcal{L}$ if and only if there is a proof of $\phi$ from $\Gamma$ in the extension $\mathcal{L}^\prime$. 
    This will be useful in the upcoming sections.
\end{remark}

\subsection{Syntactic lemmas}
For our completeness result, the following lemmas concerning the syntactic calculus are required.

\subsubsection{Properties of the disjunction}

\begin{lemma}
    \label{weak_disyuncion}
    Let $\alpha$ be a sentence, $\psi$ and $\beta$ formulas and $\Gamma$ a set of formulas.
    $$
    \text{If $\Gamma , \beta  \vdash \psi$, then $\Gamma , \alpha \vee \beta  \vdash \alpha \vee \psi$}.
    $$
\end{lemma}
\begin{proof}
Let us proceed using induction on the proof length of $\psi$ from $\Gamma, \beta$.

In the case where the proof length is $1$, it follows that $\psi$ is either an axiom, a member of $\Gamma$, or equal to $\beta$. 
\begin{itemize}
\item Suppose $\psi$ is an axiom. 
We have that $\vdash \psi \to (\alpha \vee \psi)$ (is a theorem of $\vdash_{BL\forall}$), then by Modus Ponens we get that $\vdash \alpha \vee \psi$. 
Monotonocity then give us $\Gamma, \alpha \vee \beta \vdash \alpha \vee \psi$. 
\item Let $\psi$ be an element of the set $\Gamma$. 
It follows that $\Gamma, \alpha \vee \beta \vdash \psi$ due to reflexivity. 
Moreover, since $\psi \to (\alpha \vee \psi)$ is a theorem, we can conclude by Modus Ponens that $\Gamma, \alpha \vee \beta \vdash \alpha \vee \psi$.        
\item Suppose $\psi = \beta$. 
A straightforward application of reflexivity yields $\Gamma ,\alpha\vee \beta  \vdash \alpha\vee \beta$.
\end{itemize}

Now, let us move on to the inductive step.
Let $\psi$ be a formula with a proof from $\Gamma ,\beta$ with length $\eta$, where $\eta > 1$. We will assume our desired result holds for ordinals smaller than $\eta$.
\begin{itemize}
\item If $\psi$ follows by Modus Ponens from formulas $\theta \to \psi$ and $\theta$. 
Then $\Gamma ,\beta $ proves $\theta \to \psi$ and $\theta $ with proof length smaller than $\eta$. 
By the inductive hypothesis $\Gamma ,\alpha\vee \beta \vdash \alpha \vee( \theta \to \psi)$ and $\Gamma ,\alpha\vee \beta \vdash\alpha \vee \theta$.
Since $(\alpha \vee \theta) \to ( (\alpha \vee (\theta \to \psi) ) \to (\alpha \vee \psi) )$ is a theorem of $\vdash_{BL\forall}$, we conclude that $\Gamma ,\alpha\vee \beta \vdash \alpha \vee \psi$. 
\item If $\psi$ follows by Generalization, then $\psi$ must be of the form $\forall x \chi$ and $\Gamma ,\beta  \vdash \chi$ with a proof length smaller than $\eta$. 
By the inductive hypothesis, we have $\Gamma ,\alpha\vee \beta  \vdash \alpha \vee \chi$, and through Generalization, we get $\Gamma, \alpha \vee \beta \vdash \forall x (\alpha \vee \chi)$.
Given that $\alpha$ is a sentence, we can apply the axiom (Lin) to derive $\Gamma , \alpha\vee \beta  \vdash  \alpha \vee \forall x \chi$.
\item If $\psi$ is a result of the Infinitary Rule, then $\psi$ must have the form $\rho \vee (\gamma \to (\gamma \& \delta))$ for some sentences $\rho ,\gamma ,\delta$ and $\Gamma ,\beta  \vdash \rho  \vee (\gamma  \to \delta ^i)$ with proof length smaller than $\eta$ for all $i$. 
By the inductive hypothesis, we have $\Gamma ,\alpha \vee \beta  \vdash \alpha \vee (\rho  \vee( \gamma  \to \delta ^i))$ for all $i$. 
Through associativity and the Infinitary Rule, we conclude that $\Gamma ,\alpha \vee \beta  \vdash \alpha \vee \psi$. 
\end{itemize}
\end{proof}

\begin{lemma}
    \label{disyuncion}
    Let $\alpha$ be a sentence, $\phi, \psi$ and $\beta$ formulas and $\Gamma$ a set of formulas.
    $$
    \text{If $\Gamma , \alpha \vdash \phi$ and $\Gamma , \beta  \vdash \psi$, then $\Gamma , \alpha \vee \beta  \vdash \phi \vee \psi$}.
    $$
\end{lemma}
\begin{proof}
Let us proceed with an induction on the proof length of the formula $\phi$.

In the case where the proof length is 1, the formula $\phi$ can either be an axiom, a member of the set $\Gamma$, or equal to the formula $\alpha$. 
The methods to prove these cases are similar to those demonstrated in the previous lemma (see Lemma \ref{weak_disyuncion}), with the exception of the case where $\phi = \alpha$. 
In this particular scenario, Lemma \ref{weak_disyuncion} addresses this case.

Let us proceed to the inductive step.
Let $\phi$ be a formula with a proof from $\Gamma ,\alpha$ with length $\eta$, where $\eta > 1$. 
We will assume our desired result holds for ordinals smaller than $\eta$.

We will focus on the case where $\phi$ is derived through Generalization, as the other cases are analogous to those already proven in the previous lemma (refer to Lemma \ref{weak_disyuncion}).
If $\phi$ is a consequence of Generalization, then $\phi$ is of the form $\forall x \chi$ for some formula $\chi$ such that $\Gamma ,\alpha \vdash \chi$ with a proof length smaller than $\eta$.
By applying generalization to $\Gamma, \beta \vdash \psi$, we derive that $\Gamma, \beta \vdash \forall \psi$, and by our inductive hypothesis, we conclude that $\Gamma, \alpha \vee \beta \vdash \chi \vee \forall x \psi$. 
Let us assume that $\psi$ contains a single free variable $x$, which makes $\forall x \psi$ a sentence.
This assumption is sufficient, as the case with more free variables is straightforward.
Therefore, by Generalization and (Lin), $\Gamma, \alpha \vee \beta \vdash \forall x \chi \vee \forall x \psi$.
Finally, as $(\forall x \chi \vee \forall x \psi) \to( \forall x \chi \vee \psi)$ is a theorem, we can deduce that $\Gamma ,\alpha \vee \beta  \vdash \forall \chi \vee \psi$.
\qedhere
\end{proof}
\begin{corollary}
    \label{prelinearity}
    Let $\alpha$ and $\beta $ be sentences, and $\phi$ a formula. If  $\Gamma , \alpha \to \beta  \vdash \phi$  and $\Gamma , \beta \to \alpha \vdash \phi$, then  $\Gamma \vdash \phi$.
\end{corollary}
\begin{proof}
    By the previous lemma, we have $\Gamma , (\alpha\to \beta)  \vee (\beta \to \alpha) \vdash \phi \vee \phi$.
    Then $\Gamma , (\alpha\to \beta)  \vee (\beta \to \alpha) \vdash \phi$ .
    And since $(\alpha\to \beta)  \vee (\beta  \to \alpha)$ is a theorem of BL$\forall$, we have that $\Gamma  \vdash \phi$.
\end{proof}

\subsubsection{A Henkin construction}

Let $\mathcal{L^\prime}$ be a predicate language and let $\Gamma $ be a set of formulas of $\mathcal{L^\prime}$. 
Let $\cl(\Gamma)$ denote the set of all formulas $\phi$ of $\mathcal{L^\prime}$ for which $\Gamma \vdash \phi $ and that exclusively contain constants present in $\Gamma$.

A \textit{theory} is any set of formulas. 
A \textit{prelinear theory} or \textit{complete theory} in the language $\mathcal{L^\prime}$ is a theory $\Gamma$ that satisfies for every pair of sentences $\alpha$ and $\beta$ that
\[
    \Gamma \vdash \alpha \to \beta \text{ or }\Gamma \vdash \beta \to \alpha.
\]
A \textit{Henkin theory} $\Gamma$ in  the language $\mathcal{L^\prime}$  is a theory $\Gamma$ that satisfies for all sentences $\forall x \chi $ that
\[\text{if} \quad \Gamma\nvdash \forall x \chi,  \quad \text{then} \quad  \Gamma\nvdash \chi\s{x}{c} \]
for some constant $c $.

\begin{lemma}
    \label{HenkinConstruction}
    Let $\Gamma $ be a theory and $\phi$ a sentence in the language $\mathcal{L}$ such that $\Gamma \nvdash \phi$. 
    Then, there exist a language $ \mathcal{L}^{\prime}$ extending $\mathcal{L}$ and a theory $\Gamma ^* $ in $\mathcal{L}^\prime$ containing $\Gamma$ such that $\Gamma ^* \nvdash \phi$  and $\Gamma ^*$  is prelinear and Henkin in $\mathcal{L}^\prime$.
\end{lemma}
\begin{proof}
    Let $\mathcal{L}^\prime$ be the predicate language obtained by adding a countably infinite number of constants to $\mathcal{L}$.
    Let $\psi_0, \psi_1, \psi_2, \ldots $ be an enumeration of all the formulas in $\mathcal{L}^\prime$. 
    We will proceed inductively, building sets $\Gamma _i$  and sentences $\phi_i$ for all natural $i$. 
    The union of these sets will be the set $\Gamma^* $ we are searching for.

    Take $\Gamma _0 = \Gamma $ and $\phi_0 =\phi $. 
    Let us assume that $\Gamma _i$  and $\phi_i$  are defined. 
    We have several cases:
    \begin{enumerate}
        \item $\psi_i$ is a sentence of the form $\forall x \chi $: 
            \begin{enumerate}
                \item if $\Gamma _i \nvdash \forall x \chi $ :\\
                    We have two possible scenarios:
                    \begin{enumerate}
                        \item if $\Gamma _i \nvdash \phi_i \vee \forall x \chi $, we define $$\text{$\Gamma _{i+1}=\cl (\Gamma _i)$ and $\phi_{i+1}= \phi_i \vee \chi \s{x}{c}$,}$$
                            where $c$ is a constant of the language that does not appear in $\Gamma _i \cup \{\phi_i, \psi_i\}$. 
                        \item if $\Gamma _i \vdash \phi_i \vee \forall x\chi $, we define $$\text{$\Gamma _{i+1}= \cl (\Gamma _i \cup \{ \phi_i \to \forall x \chi  \}) $ and $\phi_{i+1}=\phi_i$.}$$
                    \end{enumerate}
                \item if $\Gamma _i \vdash \forall x \chi $, we define $$\text{$\Gamma _{i+1}=\cl (\Gamma _i)$  and $\phi_{i+1} = \phi_i$.}$$
            \end{enumerate}
        \item $\psi_i$ is a sentence of the form $\gamma \vee (\alpha \to (\alpha \&  \beta))$: 
            \begin{enumerate}
                \item if $\Gamma _i, \psi_i \nvdash \phi_i$, we define $$ \text{$\Gamma _{i+i}=\cl (\Gamma _i \cup \{\psi_i \})$  and $\phi_{i+1} = \phi_i$.}$$
                \item if $\Gamma _i,\psi_i \vdash \phi_i$, we define $$\text{$\Gamma _{i+1}=\cl (\Gamma _i) $  and $ \phi_{i+1} = \phi_i \vee \gamma \vee (\alpha \to \beta^k)$,}$$ where $k$ is such that $\Gamma _i \nvdash \phi_i \vee \gamma \vee (\alpha \to \beta ^k)$. The existence of this $k$ is given by claim 1 below.
            \end{enumerate}
        \item $\psi_i$ is not of the form of items 1 and 2:
            \begin{enumerate}
                \item if $\Gamma _i, \psi_i \vdash \phi_i$, we define $$\text{$\Gamma _{i+1}=\cl (\Gamma _i)$ and $\phi_{i+1}=\phi_i$.}$$
                \item if $\Gamma _i ,\psi_i \nvdash \phi_i $, we define $$\text{$\Gamma _{i+1}=\cl (\Gamma _i \cup \{ \psi_i \})$ and $\phi_{i+1}=\phi_i$.}$$ 
            \end{enumerate}
    \end{enumerate}

    \begin{claim}
        For every natural number $i $:
        \begin{enumerate}
            \item $\Gamma_i $ only contains a finite number of new constants,
            \item $\Gamma_i \nvdash \phi_i $, 
            \item Under the conditions of 2.(b), there exists a natural number $k $ such that
                \[
                    \Gamma_i \nvdash \phi_i \vee \gamma \vee (\alpha \to \beta^k).
                \]
        \end{enumerate}
    \end{claim}
    \begin{proof}
        We will proceed by induction on $i$. The case $i=0$ is immediate (we can assume that $\psi_0$ is chosen such that $\Gamma_0, \psi_0 \nvdash \phi_0$).

        Now let us suppose the result is true for a natural $i $. 
        \begin{enumerate}
            \item We have that $\Gamma_{i+1} $ can be one of the following
                \[
                    \cl(\Gamma_i), \cl(\Gamma_i \cup \{ \phi_i \to \psi_i \}), \cl(\Gamma_i \cup \{ \phi_i  \}).
                \]
                Since $\phi_i $ and $\psi_i $ only contain a finite number of constants, $\Gamma_i $ only contains a finite number of new constants by hypothesis, and $\cl $ does not add new constants, we can conclude that $\Gamma_{i+1} $ contains only a finite number of new constants. 

            \item We prove that $\Gamma _{i+1}\nvdash \phi_{i+1}$ in case $\psi_i$ is a sentence of the form $\forall x \chi $ with $\Gamma _i \nvdash \psi_i $ and $\Gamma _i \vdash \phi_i \vee \forall x \chi $; the rest of the cases being straightforward. 
                In this case the construction tells us that $\Gamma _{i+1}=\cl(\Gamma _i \cup \{\phi_i \to \forall x \chi\})$ and
                    $\phi_{i+1}=\phi_i$. Then, all we must prove is $\Gamma _i, \phi_i \to \forall x \chi  \nvdash \phi_i$.

                    Let us suppose that this is false, that is, $\Gamma _i, \phi_i \to \forall x \chi  \vdash \phi_i$. 
                    Observe that, by hypothesis, we have that $\Gamma _i\vdash \phi_i \vee \forall x \chi $, so, since $(\forall x \chi \to \phi_i)\to ((\phi_i \vee \forall x \chi  ) \to \phi_i))$ is a theorem, we infer $\Gamma _i, \forall x \chi  \to \phi_i \vdash \phi_i$.

                    Thus, we have that $\Gamma _i, \forall x \chi  \to \phi_i \vdash \phi_i$ and that $\Gamma _i, \phi_i \to \forall x \chi  \vdash \phi_i$. Since both $\forall x \chi$ and $\phi_i$ are sentences by construction, according to Corollary \ref{prelinearity} we can deduce that $\Gamma _i\vdash \phi_i$, which contradicts the inductive hypothesis.

            \item Assume the conditions of 2.(b) and that there is no $k $ such that $\Gamma _i \nvdash \phi_i \vee \gamma \vee (\alpha \to \beta^k) $. 
                By the Infinitary Rule, we see that $\Gamma _i \vdash (\phi_i \vee \gamma) \vee (\alpha \to (\alpha \& \beta))$, that is, $\Gamma _i \vdash \phi_i \vee \psi_i $. 

                Recall that by Reflexivity we have $\Gamma _i, \phi_i \vdash \phi_i$, and, by hyphotesis, we have $\Gamma _i, \psi_i \vdash \phi_i $, then using Corollary \ref{disyuncion} we obtain $\Gamma _i, \phi_i \vee \psi_i \vdash \phi_i$, and since $\Gamma _i \vdash \phi_i \vee \psi_i$ we can assert that $\Gamma _i \vdash \phi_i $, contradicting the inductive hypothesis.

            \end{enumerate}

    \end{proof}

    The claim just proved shows that the steps 1.(a).i and 2.(b) of the construction are sound.

    \begin{claim}
        For all natural numbers $i,j$ such that $i \leq j$ we have that
        \begin{enumerate}
            \item $\Gamma_i \subseteq \Gamma_j$,
            \item $\vdash \phi_i \to \phi_j$.
        \end{enumerate}
    \end{claim}
    \begin{proof}
        Immediate from the construction.
    \end{proof}

    \begin{claim}
        Let $ \Gamma^* = \bigcup_i \Gamma_i$. If $\Gamma^* \vdash \theta $ for a formula $\theta$, then $\theta \in \Gamma^* $.
    \end{claim}
    \begin{proof}
        Let us proceed by induction on the proof length of $\theta$. For the initial case we have that $\theta$ is an axiom or it belongs to $\Gamma^* $. We leave these cases to the
        careful reader.

        Let us consider a proof of $\theta $ from $\Gamma^* $ with length $\eta $. We have 3 cases:    
        \begin{itemize}
            \item $\theta$ follows by Modus Ponens from $\gamma$ and $\gamma \to \theta$. 
                We then have that $\Gamma^* $  proves $\gamma$ and $\gamma \to \theta $ with proof lengths smaller than $\eta$.
                Then, by the inductive hypothesis, we obtain $\gamma \in \Gamma^* $ and $\gamma \to \theta \in \Gamma^* $.
                We thus get that there exists $i$ such that $\gamma \in \Gamma _i$ and $\gamma \to \theta \in \Gamma _i $.
                Since $\cl(\Gamma_i) \subseteq \Gamma _{i+1}$ we obtain that $\theta \in \Gamma _{i+1}$. We can then
                conclude that $\theta \in \Gamma^*$.

            \item $\theta$  follows by Generalization. This case can be proved in a similar fashion to the previous case.

            \item $\theta$ follows by the Infinitary Rule.
                Then, $\theta$ must be of the form $\gamma \vee (\alpha \to (\alpha \&  \beta))$ for some sentences $\alpha, \beta, \gamma$.

                Let us suppose that $\theta \notin \Gamma^* $. Let $i$ be such that $\psi_i = \theta$. Then, $\psi_i \notin \Gamma _{i+1}$, thus, by construction, we have that
                \[
                    \Gamma _i, \psi_i \vdash \phi_i \text{ and } \phi_{i+1}= \phi_i \vee \gamma \vee (\alpha \to \beta^k) 
                \]
                for some $k $. 
                Since $\psi_i$ follows by the Infinitary Rule, we have that $\Gamma^* $ proves $\gamma \vee (\alpha \to \beta^k)$ with proof length smaller than $\eta$. Then, by the inductive hypothesis, $\gamma \vee (\alpha \to \beta ^k) \in \Gamma^* $. 
                Thus, there exists $j$ such that $\gamma \vee (\alpha \to \beta ^k) \in \Gamma _{j}$, so $\Gamma_{j } \vdash \phi_i \vee \gamma \vee (\alpha \to \beta ^k)$, that is, $\Gamma_j
                \vdash \phi_{i+1}$. We have two cases:
                \begin{description}
                    \item[$j \leq i+1$:] By Claim 2 we have that $\Gamma_j \subseteq \Gamma_{i+1} $, then by monotonicity of $\vdash $ we obtain $ \Gamma_{i+1} \vdash \phi_{i+1}$ contradicting Claim 1.
                    \item[$j > i+1$:] By Claim 2 we have that $\vdash \phi_{i+1} \to \phi_j$, then $\Gamma_j \vdash \phi_j$, contradicting Claim 1.
                \end{description}
        \end{itemize}
    \end{proof}

    \begin{claim}
        $\Gamma^* \nvdash \phi_i $ for all $i $.
    \end{claim}
    \begin{proof}
        Let us suppose $\Gamma^* \vdash \phi_i$. By the Claim 3, we have $\phi_i \in \Gamma^* $. 
        This shows that there exists $j$  such that $\phi_i \in \Gamma _j$.
        Then $\Gamma_j \vdash \phi_i$. We can proceed from this point and obtain a contradiction in a similar way as what is done at the end of the proof of the previous claim.
    \end{proof}

    As an immediate consequence of the previous claim, we have that $\Gamma^* \nvdash \phi$.

    \begin{claim}
        $\Gamma^* $ is a Henkin Theory. 
    \end{claim}
    \begin{proof}
        Let $\forall x \chi $  be a sentence such that $\Gamma^* \nvdash  \forall x \chi $. We want to prove that there exists a constant $c$  such that $\Gamma^* \nvdash \chi \s{x}{c}$. 

        Let $i$  be such that $\psi_i = \forall x \chi  $. 
        We then have three cases: 
        \begin{itemize}
            \item $\Gamma _i \vdash \psi_i$:\\
                This contradicts $\Gamma^*  \nvdash \forall x \chi  $.
            \item $\Gamma _i \nvdash \psi_i $ and $\Gamma _i \vdash \phi_i \vee \forall x \chi $ :\\
                In this case, the construction tells us that $\phi_i \to \forall x \chi \in \Gamma _{i+1} $. 

                Then, since $(\phi_i \vee \forall x \chi) \to ((\phi_i \to \forall x \chi) \to \forall x \chi)$  is a theorem, we deduce that $\Gamma _{i+1}\vdash \forall x \chi$.
                Thus, $\Gamma^* \vdash \forall x \chi $, a contradiction. 
            \item $\Gamma _i \nvdash \psi_i$  and $\Gamma _i\nvdash \phi_i \vee \forall x \chi $:\\
                In this case, the construction tells us that $\phi_{i+1}=\phi_i \vee \chi\s{x}{c}$ where $c$ is a constant such that does not appear in $\Gamma _i \cup \{\phi_i, \psi_i \}$). By Claim 4 we get that $\Gamma^* \nvdash \phi_i \vee \chi\s{x}{c}$. Then, $\Gamma^* \nvdash \chi\s{x}{c}$, since $\chi\s{x}{c} \to (\phi_i \vee \chi\s{x}{c})$  is a theorem. 
        \end{itemize}
    \end{proof}

    \begin{claim}
        $\Gamma^* $ is prelinear.
    \end{claim}
    \begin{proof}
        Let $\alpha ,\beta $  be sentences. 
        Let $i,j$  be such that $\alpha \to \beta =\psi_i$  and $\beta  \to \alpha  = \psi_j$.
        Suppose that $\Gamma^* \nvdash \psi_i$ and that $\Gamma^* \nvdash \psi_j$ . 
        This tells us that $\psi_i \notin \Gamma^* $ and that $\psi_j \notin \Gamma^* $.
        Then, by construction, we obtain $\Gamma _i, \psi_i \vdash \phi_i$ and $\Gamma _j, \psi_j \vdash \phi_j$.
        Let $l = \max(i,j)$. By Claim 2 we deduce that $\Gamma _l, \psi_i \vdash \phi_l$ and $\Gamma _l, \psi_j \vdash \phi_l$. 
        Then, since $\psi_i= \alpha \to \beta $ and $\psi_j = \beta \to \alpha $, and using Corollary \ref{prelinearity}, we conclude that $\Gamma _l \vdash \phi_l$, contradicting Claim 1.
    \end{proof}

    This concludes the proof of the theorem.
\end{proof}

\subsection{Soundness}
In this section, we quickly prove that the infinitary inference rule is sound in countinuous t-norms. This, combined with the fact that both the axiom RC and the logic $\vdash_{BL}$  are sound in continuous t-norms \cite{Hajek98}, will lead to a soundness result for our logic.
\begin{lemma} \label{infinitaryRuleSoundness}
    Let $\phi$, $\alpha$ and $\beta$ be formulas.
    Then $$\{\phi \vee (\alpha \to \beta^n) : n\in \mathbb{N} \}\vDash \phi \vee (\alpha \to \alpha \& \beta).$$
\end{lemma}
\begin{proof}
    Consider a continuous t-norm $\mathbf{B}=([0,1], \cdot)$ and a $\mathbf{B}$-structure $\mathbf{M}$ such that for each $n \in \mathbb{N}$ we have that $\mathbf{M}$ is a model of the formula $\phi \vee (\alpha \to \beta^n)$. 
    Let $v$ be a valuation.  
    Our goal is to prove that $\| \phi \vee (\alpha \to \alpha \& \beta) \|^{\mathbf{M},v}=1$.
    Given that  $\mathbf{M}$ is a model of $\phi \vee (\alpha \to \beta^n)$ for all $n\in \mathbb{N}$, we have that $\| \phi \|^{\mathbf{M},v} \vee ( \| \alpha \|^{\mathbf{M},v} \to (\| \beta \|^{\mathbf{M},v})^n) = \| \phi \vee (\alpha \to \beta^n) \|^{\mathbf{M},v } = 1$ for all $n\in \mathbb{N}$.
    For convenience let us call  $\| \alpha \|^{\mathbf{M},v} = a$, $\| \beta \|^{\mathbf{M},v} = b$ and $\| \phi \|^{\mathbf{M},v} = c$, thus, replacing in the last expresion, we obtain $c \vee ( a \to  b ^n) = 1$.
    
    If $c = 1$, it is evident that $1 = c \vee (a \to a \cdot b) = \| \phi \vee (\alpha \to \alpha \& \beta) \|^{\mathbf{M},v}$. 
    
    If $c < 1$, since $\mathbf{B}$ is totally ordered, we have $a \to b^n= 1$ for all $n$, implying $a \leq b^n$ for all $n$, and then
    $$ 
    a \wedge \bigwedge_{n} b^n = a.
    $$
    Now, replacing, we obtain
    $$
    a \cdot b= (a \wedge \bigwedge_{n} b^n) \cdot b = (\bigwedge_{n} b^n \cdot (\bigwedge_{n} b^n \to a)) \cdot b.
    $$ 
    As $\cdot $ is continuous, we have $(\bigwedge_{n} b^n) \cdot  b = \bigwedge_{n} b ^{n+1}= \bigwedge_{n} b^n$. 
    Therefore 
    $$
    a \cdot  b =(\bigwedge_{n} b^n \cdot (\bigwedge_{n} b^n \to a)) \cdot b = \bigwedge_{n} b^n \cdot  (\bigwedge_{n} b^n \to a) = a \wedge \bigwedge_{n} b^n = a.
    $$ 
    Thus $1 = a \to (a \cdot b)= c \vee (a \to (a \cdot b)) = \| \phi \vee (\alpha \to  (\alpha \& \beta)) \|^{\mathbf{M},v}$. 
\end{proof}
\begin{corollary}
The logic $\vdash$ is sound with respect to the class of continuous t-norms.
\end{corollary}

\section{Completeness}
\label{sectionCompleteness}
Let $\mathbf{Sen}=(Sen, \cdot, \to, \wedge, \vee, 0, 1)$ be defined in the natural way, where $\phi \cdot \psi = \phi \& \psi$.
Given a theory $\Gamma \subseteq Fm$, let $\equiv_{\Gamma}$ be the equivalence relation on $Sen$ given by 
\begin{equation*}
    \phi \equiv_{\Gamma} \psi \iff \Gamma \vdash \phi \to \psi \text{ and } \Gamma \vdash \psi \to \phi.
\end{equation*}
We write $[\phi]_{}$ for the equivalence class of $\phi$ modulo $\equiv_{\Gamma}$. 
It can be shown that this relation is a congruence of $\mathbf{Sen}$.
Therefore, we can define an algebra $\mathbf{L}_\Gamma = \sfrac{\mathbf{Sen}}{\equiv_{\Gamma}}$.
We call this algebra \textit{Lindenbaum algebra}. 

We reference Lemma 5.2.6 from \cite{Hajek98} for future use.
\begin{lemma}[Hájek] \label{Henkin}
\begin{enumerate} 
    \item If $\Gamma$ is a Henkin theory then 
    \begin{equation*}
    [\forall x \phi]= \inf_c [\phi\s{x}{c}] \text{ and } [\exists x \phi]= \sup_c [\phi\s{x}{c}],
    \end{equation*}
    where $c$ runs over all constants in $\Gamma$. 
    \item If $\Gamma $ is a prelinear theory then $\textbf{L}_{\Gamma}$ is a chain, that is, a totally ordered poset.  
\end{enumerate}
\end{lemma}

Let then $\Gamma$ be a prelinear theory and $\mathrm{\mathbf{L}}$ the associated Lindenbaum algebra. 
We now provide some lemmas that allow us to embed this algebra in a t-norm. 

Remember that an algebra is called \textit{simple} when its only filters are trivial.
\begin{lemma} \label{infi_implies_simple}
    Any BL-chain that satisfies the generalized quasi-identity
\[
\tag{K}
    (\&_{n\in \mathbb{N}} (\alpha \to \beta^n) = 1 ) \Rightarrow (\alpha \to \alpha \& \beta = 1) 
\] 
is an ordinal sum of simple hoops. 
\end{lemma}
\begin{proof}
Lemma \ref{decomposition} in the introduction states that every BL-chain can be represented as an ordinal sum of totally ordered Wajsberg hoops.
If a BL-chain satisfies (K), then the hoops in the decomposition also satisfy (K).
We will now show that every totally ordered Wajsberg hoop satisfying (K) is simple.

Let $\mathbf{H}$ be a Wajsberg hoop. 
Consider a proper filter $F$ (non trivial) of $\mathbf{H}$, with $b \in F$ and $a \in H-F$.
Since $a \leq b^n$ for all $n \in \mathbb{N}$, by applying (K), we have $a \leq a \cdot b$.
This implies $a = a \cdot b$.

Now, we consider whether $\mathbf{H}$ is bounded or not.
If $\mathbf{H}$ is unbounded, then $\mathbf{H}$ is a cancellative hoop \cite{BF00}, and from $a = a\cdot b$, we get $1 = b$.
Therefore, if $\mathbf{H}$ is unbounded, there are no proper filters, making $\mathbf{H}$ simple.

If $\mathbf{H}$ is bounded, it has a minimum element $0$.
Again, considering $a = a \cdot b$, we have:
\begin{align*}
0 &= a \cdot (a \to 0)\\
 &= a \cdot ((a \cdot b) \to 0) \\
&= a \cdot ( a \to (b \to 0)) \\
&= a \wedge (b \to 0).
\end{align*}
This leads to the conclusion that $a=0$ or $b=1$ (since  $b=(b\to 0)\to 0 = 0\to 0 =1$).
Thus, the only possible element outside $F$ is $0$, that is $F=H-\{0\}$.
Now, consider $u\in F$, then $u\to 0\notin F$, meaning $u\to 0 =  0$, leading to $u = 1$. 
This implies $F$ is a trivial filter, making $\mathbf{H}$ simple.
\end{proof}  

We call an element $a $ in an BL-algebra \textit{idempotent} iff $a \cdot a = a$.

A pair $(X,Y)$ is called a \textit{cut} in a BL-chain $A$ if the following conditions hold:
\begin{itemize}
\item $X \cup Y = A$,
\item $x\leq y$, for all $x\in X$ and $y \in Y$,
\item $Y$ is closed under $\cdot$, and 
\item $x*y=x$, for all $x \in X$ and $y \in Y$.
\end{itemize}
A cut in a BL-chain $A$ is called \textit{saturated} if there is an idempotent $u\in A$ such that $x \leq u \leq y$ for all $x\in X$ and $y \in Y$.
A BL-chain $A$ is called \textit{saturated} if each cut $(X,Y)$ in $A$ is saturated.

A BL-chain $A$ is called \textit{weakly saturated} if the following holds:
\begin{itemize}
\item If $c\in A$ is not an idempotent, then there are idempotents $a,b$ such that $a < c < b$ and the set 
\[
    [a,b] = \{ x\in A : a \leq x \leq b \}
\]
is the domain of either an MV-algebra or of a product algebra, denoted by $\mathbf{A}_{[a,b]}$, with respect to the operations
\begin{equation} \label{eq:defining_algebra_cluster}
    x *_{[a,b]} y = x * y \text{ and } x \to_{[a,b]} y = \min(x\to y, b).
\end{equation}
Moreover, $\mathbf{A}_{[a,b]}$ is Archimedean, that is, if $a < r < s < b$, then there is $n$ such that $s^n < r$.
\item If $c \in A$ is an idempotent, then there are $a \leq c$ and $b \geq c$ such that $[a,b]$ entirely consists of idempotents and 
\begin{itemize}
\item for all $h <a$ there is $h'$ such that $h\leq h' < a$ and $h'$ is not an idempotent;
\item for all $k>b$ there is $k'$ such that $b<k' \leq k$ and $k'$ is not an idempotent; thus if $a<b$, then $\mathbf{A}_{[a,b]}$ defined as in \eqref{eq:defining_algebra_cluster} is a Gödel algebra.
\end{itemize}
\end{itemize}

\begin{lemma} \label{oldSaturate}
Each BL-chain $\mathbf{A}$ can be isomorphically embedded into a saturated BL-chain $\overline{\mathbf{A}}$. Moreover:
\begin{enumerate}
\item the new elements are all idempotents,
\item if $C \subseteq A$ and $\bigwedge^{\mathbf{A}}C$ exists and is idempotent, then $\bigwedge^{\overline{\mathbf{A}}}C$ exists and $\bigwedge^{\mathbf{A}}C= \bigwedge^{\overline{\mathbf{A}}}C$. The same holds true for the supremum.
\end{enumerate}
\end{lemma}
\begin{proof}
This result is based on Theorem 4 in \cite{CT05}, but it does not prove (2). We will provide the proof for this claim.

Let $\mathbf{A}$ be a BL-chain and $S$ be the set of all non-saturated cuts of $\mathbf{A}$. 
For each $\alpha \in S$, the upper part of $\alpha$ will be denoted $\alpha^+$, and the lower part, by $\alpha^-$. 
On $\overline{A} = A \cup S$ define the binary relation $\leq$ as follows:
\begin{equation*}
x \leq y \iff
\begin{cases}
x,y \in A \text{ and } x \leq^{\mathbf{A}} y,  \\
x,y \in S \text{ and } x^- \subseteq y^-,  \\
x \in A, y\in S \text{ and } x \in y^-,  \\
x \in S, y\in C \text{ and } y \in x^+.
\end{cases}
\end{equation*}
It follows that $\leq$ is a total order in $\overline{A}$.
Operations are extended naturally.
For more details, refer to \cite{CT05}.

Let $C \subseteq A$ such that $\bigwedge^{\mathbf{A}}C$ exists and is idempotent.
In the extension $\overline{\mathbf{A}}$, new elements not present in $\mathbf{A}$ may be inserted between $\bigwedge^{\mathbf{A}} C$ and $C$. 
Let $H$ denote this set of new elements.
Take an element $h$ in $H$.
Since $h$ is a new element, it must be an unsaturated cut of $\mathbf{A}$, meaning the sets $\{ x \in A \mid x \leq h \}$, $\{ x \in A \mid h \leq x\}$ form an unsaturated cut of $\mathbf{A}$.
However, the idempotent element $\bigwedge^{\mathbf{A}}C$ already saturates this cut, leading to a contradiction.
Therefore, $H$ must be empty, implying $\bigwedge^{\mathbf{A}}C= \bigwedge^{\overline{\mathbf{A}}}C$.
\end{proof}

If $\mathbf{A}=\bigoplus_{i\in I}\mathbf{H}_i$ is a BL-chain represented as an ordinal sum of Wajsberg hoops and $C \subseteq A$ is a subset such that $\bigwedge C$ exists, then we define $C$ to be $\wedge$-\textit{connected} if there exists $i \in I$ and $c\in C$ such that $\bigwedge C \in H_i$ and $c \in H_i$. 
The concept of $\vee$-connected is defined dually.

\begin{lemma} \label{weakly_saturate}
Let $\mathbf{A}=\bigoplus_{i\in I}\mathbf{H}_i$ be a countable BL-chain that is an ordinal sum of totally ordered Wajsberg simple hoops.
Then there is an embedding $f$ from $\mathbf{A}$ into a countable weakly saturated BL-chain $\mathbf{A}^*$. 
Also, if $C$ is a subset of $A$ such that $\bigwedge^{\mathbf{A}} C$ exists and either $C$ is $\wedge$-connected or $\bigwedge^{\mathbf{A}}C$ is idempotent, then $f(\bigwedge^{\mathbf{A}}C )= \bigwedge^{\mathbf{A}^*}f(C)$. The same holds true for the supremum.
\end{lemma}
\begin{proof}
Let $\mathbf{A}= \bigoplus_{i\in I}\mathbf{H}_i$ be a countable BL-chain represented as an ordinal sum of totally ordered Wajsberg simple hoops.
By Lemma \ref{oldSaturate} we can embed $A$ in a saturated BL-chain $\overline{A}$.
From Lemma 3 in \cite{CT05}, we know that $\overline{\mathbf{A}}$ is weakly saturated. 
Then for each $c \in \overline{A }- \text{idp}(\overline{A})$, there exist elements $a_c,b_c \in \text{idp}(\overline{A})$ such that $c \in [a_c,b_c]$ and $\overline{\mathbf{A}}_{[a_c,b_c]}$ is either a MV-algebra or a product algebra, and for each $c \in \idp(\overline{\mathbf{A}})$, we have elements $a_c,b_c \in \idp(\overline{A})$ such that $c \in [a_c,b_c]$ and $\overline{\mathbf{A}}_{[a_c,b_c]}$ is a maximal Gödel algebra.

Let $S=\{ a_c,b_c \mid c\in A \}$, a countable set since $A$ is countable. 
All elements of $S$ are idempotents, making $\mathbf{S \cup A}$ a subalgebra of $\overline{\mathbf{A}}$.
It is also clear that $\mathbf{S \cup A}$ is weakly saturated.

Define the set $S=\{ a_c,b_c \mid c\in A \}$. 
Since $A$ is countable, $S$ is countable. 
All elements of $S$ are idempotents, then $\mathbf{S \cup A}$ is a subalgebra of $\overline{\mathbf{A}}$.
It is easy to see that the algebra $\mathbf{S \cup A}$ is weakly saturated. 

Consider now a subset $C \subseteq A$ such that $\bigwedge^{\mathbf{A}}C$ exists.
First, suppose $C$ is $\wedge$-connected. 
Then, there exists $i\in I$ and $c\in C$ such that $\bigwedge^{\mathbf{A}}C \in H_i$ and $c \in H_i$.
If $c$ is not idempotent, it can be readily observed that the set $\{ x \in S \mid \bigwedge^{\mathbf{A}}C < x < c\}$ is empty, given the fact that the only possible idempotent element in $H_i$ is its first element.
This implies that $\bigwedge^{\mathbf{A}}C = \bigwedge^{\mathbf{A}^*}C$.
If $c$ is idempotent, it must be the first element of $H_i$, leading to $\bigwedge^{\mathbf{A}}C = \bigwedge^{\mathbf{A}^*}C = c$.
Now, suppose $\bigwedge^{\mathbf{A}}C$ is idempotent. 
It is a consequence of Lemma \ref{oldSaturate} that $\bigwedge^{\mathbf{A}}C = \bigwedge^{\mathbf{A}^*}C$. 
The proof of the supremum case can be demonstrated using a similar approach.

\end{proof}

\begin{lemma} \label{inf_RC_idempotente}
Let $\mathbf{A}$ be a BL-chain and $\{a_x\}_{x\in X}$ a sequence in $A$.
\begin{itemize}
\item If $(\inf_x a_x)^2 = \inf_x a_x^2$, then $\{a_x\}$ is $\wedge$-connected or $\inf_x a_x$ is idempotent. 
\item If the supremum of $\{a_x\}$ exists, then $\{a_x\}$ is $\vee$-connected or $\sup_x a_x$ is idempotent. 
\end{itemize}
\end{lemma}
\begin{proof}
Let $\mathbf{A}=\bigoplus_{i\in I}\mathbf{H}_i$ be a BL-chain represented as an ordinal sum of totally ordered Wajsberg hoops.

1) Let $y\in X$. 
Then $a_y \in H_i$ and $\inf_x a_x \in H_j$ for some $i,j \in I$, and consequently $a_y^2 \in H_i$.
If the sequence $\{ a_x \}$ is not $\wedge$-connected, then $j < i$, implying $\inf_x a_x < a_y^2$.
Since $y$ was chosen arbitrarily, we have $\inf_x a_x \leq \inf_x a_x^2$.
The other inequality is trivial.
Thus, $\inf_x a_x = \inf_x a_x^2 = (\inf_x a_x)^2$.

2) We have that $\sup_x a_x \in H_i$ for some $i$. 
Let $b \in H_i$.
If the sequence $\{ a_x \}$ is not $\vee$-connected, then $b > a_x$ for all $x\in X$, implying $b \geq \sup_x a_x$.
Therefore, $\sup_x a_x$ is the minimum element of $H_i$, and thus, an idempotent element.
\end{proof} 

Let $A $ and $B$ be posets. Recall that a function $f:A\to B$ is a \textit{complete embedding} if and only if $\inf_{s\in S} f(s) = f(\inf_{s\in S} s)$ and $\sup_{t\in T} f(t) = f(\sup_{t\in T} t)$ whenever $\bigwedge S $ exists in $A$ and $\bigvee T $ exists in $A$. 

The following result comes from \cite{HajekMontagna08}, mostly from Lemma 3.6 and Lemma 3.7.
\begin{lemma} \label{weakly_en_tnorma}
Every countable and weakly saturated BL-chain $\mathbf{A}$ embeds into a continuous t-norm by a complete embedding.
\end{lemma}

\begin{theorem}[Completeness]
Let $\Gamma$ be a theory and $\phi$ a sentence. Then
\[
\Gamma \vdash \phi \iff \Gamma \vDash \phi.
\] 
\end{theorem}
\begin{proof}
We only have to prove $\Gamma \vDash \phi \implies \Gamma \vdash \phi$.
Let us assume we have a theory $\Gamma$ and a sentence $\phi$ such that $\Gamma \nvdash \phi$.
By Lemma \ref{HenkinConstruction}, we can establish the existence of a theory $\Gamma ^*$ that contains $\Gamma$, satisfies $\Gamma^* \nvdash \phi$, and is both prelinear and Henkin. 

Next, let $\mathbf{L}$ be the Lindenbaum algebra  asociated with $\Gamma ^ *$.
By further applying Lemma \ref{Henkin}, we deduce that $\mathbf{L}$ is totally ordered. 
Since the logic $\vdash$ has the infinitary rule (\ref{Inf}), it is clear that $\mathbf{L}$ satisfies the generalized quasi-identity $(\&_{n\in \mathbb{N}} (\alpha \to \beta^n) = 1) \Rightarrow (\alpha \to \alpha \& \beta = 1)$. 
Then by Lemma \ref{infi_implies_simple}, we have that $\mathbf{L} = \bigoplus_{i\in I}\mathbf{H}_i $, where $I$ is a totally ordered poset with a minimum and each $\mathbf{H}_i$ is a simple hoop. 

Given that the language is countable, $L$ is also countable. This allow us to apply Lemma \ref{weakly_saturate} to obtain an embedding $f:\mathbf{L}\to \mathbf{L^*}$, where $\mathbf{L^*}$ is a countable weakly saturated BL-chain and if $C \subseteq A$ is $\wedge$-connected or if $\bigwedge^{\mathbf{A}}C$ is idempotent, then $\bigwedge^{\mathbf{A}}C = \bigwedge^{\mathbf{A}^*}C$ (the same holds for the supremum).
Subsequently, according to Lemma \ref{weakly_en_tnorma}, there exists a complete embedding $g$ from $\mathbf{L^*}$ into a continuous t-norm $([0,1], \cdot)$.

We now proceed to construct a model based on $([0,1], \cdot)$.
Let $M$ to be the set of all constants of the language. 
For each constant $c$ we define $c^\mathbf{M} = c$.
For each predicate symbol $R$ of arity $n$ we define $R^{\mathbf{M}}(c_1, \ldots, c_n) = g(f([R(c_1, \ldots, c_n)]))$.
This completes the definition of $\mathbf{M} = (M,.^\mathbf{M})$.

Our goal is to prove that the value $g(f([\psi]))$ of any sentence $\psi$ is equal to its interpretation in the structure $\mathbf{M}$, that is $g(f([\psi])) = || \psi ||^{\mathbf{M}}$.
We proceed by induction on the structure of the sentence $\psi$.
For any atomic formula the claim follows by definition, and the induction step for connectives is straightforward.
Consider a universally quantified sentence $\psi = \forall x \chi$. By Lemma \ref{Henkin} we have
\[
g(f([\forall x \chi])) = g(f( \inf_c [\chi\s{x}{c}] )).
\]
Given that (\ref{RC}) is an axiom and Lemma \ref{Henkin}, it follows that $(\inf_c [\chi\s{x}{c}])^2 = \inf_c [\chi\s{x}{c}]^2$. 
Knowing this, Lemma \ref{inf_RC_idempotente} implies that the set $\{\chi\s{x}{c}\}_c$ is either $\wedge$-connected or $\inf_c [\chi\s{x}{c}]$ is idempotent. Therefore, as $f$ preserves infimum on sets that are $\wedge$-connected or yield idempotent elements, 
\[
g(f( \inf_c [\chi\s{x}{c}] )) =g( \inf_c f([\chi\s{x}{c}] )),
\]
and since $g$ is complete,
\[
g( \inf_c f([\chi\s{x}{c}] )) = \inf_c g(f([\chi\s{x}{c}] )).
\]
Applying the inductive hypothesis and the definition of the interpretation of the infimum completes the proof for the universally quantified sentences. 
The case where $\psi$ is an existentially quantified sentence follows in the same way, except that the axiom (\ref{RC}) is not required.

We have $[\psi]=1$ for all $\psi \in \Gamma$ and, since $\Gamma^* \nvdash \phi$, we also have $[\phi] < 1$. 
Then $||\psi||^{\mathbf{M}} = g(f([\psi]))= 1$ for all $\psi \in \Gamma$ and $||\phi||^{\mathbf{M}} = g(f([\phi])) < 1$. 
Thus $\Gamma \nvDash \phi$.
\end{proof}  

As an easy corollary of the completeness result and Lemma \ref{infinitaryRuleSoundness} we have that the infinitary rule, which require the formulas involved to be sentences, can now be modified to allow formulas that are not necessarily sentences.

\bibliographystyle{plain}
\bibliography{version_saturadas}

\end{document}